\documentclass[submission,copyright,creativecommons]{eptcs}
\providecommand{\event}{Theoretical Aspects
of Rationality and Knowledge 2023 (TARK 2023)} 
\usepackage{amsmath}
\usepackage{amssymb}
 \usepackage{amsthm}
 \usepackage{amsfonts}
 \usepackage{xcolor}
 \usepackage{mathrsfs}
\usepackage{stmaryrd}
\usepackage{enumitem}
\usepackage{thmtools, thm-restate}

\usepackage[mathcal]{euscript}

\newtheorem{theorem}{Theorem}
\newtheorem{lemma}[theorem]{Lemma}

\newtheorem{proposition}[theorem]{Proposition}
\newtheorem{corollary}[theorem]{Corollary}

\theoremstyle{definition}
\newtheorem{definition}[theorem]{Definition}

\newcommand\xqed[1]{%
  \leavevmode\unskip\penalty9999 \hbox{}\nobreak\hfill
  \quad\hbox{{#1}}}
\newcommand\qeddef{\xqed{${\scriptstyle\blacktriangleleft}$}}
\newcommand\qedthm{\xqed{${\scriptstyle\blacksquare}$}}

\title{Strengthening Consistency Results
in Modal Logic}
\author{Samuel Allen Alexander
\institute{US Securities and Exchange Commission\\New York, USA}
\email{samuelallenalexander@gmail.com}
\and
Arthur Paul Pedersen
\institute{City University of New York\\
New York, USA}
\email{apedersen@cs.ccny.cuny.edu}
}
\def\titlerunning{Strengthening Consistency Results in Modal Logic}
\def\authorrunning{S.A. Alexander \& A.P. Pedersen}
\begin{document}
\maketitle





\begin{abstract}
 A fundamental question asked in modal logic is whether a given theory
    is consistent. But consistent with what? A typical way to address this question identifies
   a choice of background knowledge axioms (say, S4, D, etc.)\ and then shows
    the assumptions codified by the theory in question to be consistent with those background axioms.
    But determining the specific choice and division of background axioms is, at least sometimes, little more than
    tradition. This paper introduces  \emph{generic theories}
    for propositional modal logic to address consistency results in a more robust way. As building blocks for background knowledge, generic theories provide a standard for categorical determinations of consistency.  
   We argue that the results and methods of this paper 
    help to elucidate problems in epistemology and enjoy sufficient scope and power to have purchase on problems bearing on modalities in judgement, inference, and decision making.
\end{abstract}
\section{Introduction}

Many treatments of epistemological paradoxes in modal logic proceed along the following lines. Begin with some enumeration of assumptions that are individually plausible  but when taken together fail to be jointly consistent (or at any rate fail to stand to reason in some way). Thereupon proceed to propose a resolution to the emerging paradox that identifies one or more assumptions that may be comfortably discarded or weakened and that in the presence of the remaining assumptions circumvents the troubling inconsistency defining the paradox \cite{haack1978philosophy}
 (cf.  Chow \cite{chow1998surprise} and de Vos et al. \cite{de2023solutions}). Typical among such assumptions are logical standards expressed in the form of inference rules and axioms pertaining to knowledge and belief, such as axiom scheme $\mathbf{K}$ --- that is to say, the distributive axiom scheme of the form $\mathrm{K}(\varphi\rightarrow\psi)\rightarrow (\mathrm{K}\varphi\rightarrow \mathrm{K}\psi)$.

 The choice of precisely \emph{which} assumptions to temper can, at times, have an element of arbitrariness to it, especially when the choice is made from among several independent alternatives underpinning distinct resolutions in the absence of clear criteria or compelling grounds for distinguishing among them.   In the present paper, we introduce a criterion for addressing this predicament based on the \emph{genericity} of what a resolution assumes.

As a standard for knowledge, a theory is generic when its factivity cannot be overturned however the questions it leaves open are answered and what is known accordingly grows. Generic theories enjoy various desirable
properties which are common in formal epistemology --- arbitrary unions of generic
theories, for example, are generic. We present both positive and negative results turning on genericness, which cast light on the structure of popular logics for belief and knowledge.

The concept of generic theories, as introduced in \cite{alexander2015fast}
and \cite{alexander2020self}  for quantified modal logic, emerged in response to Carlson's proof \cite{carlson2000knowledge} of a conjecture due to
Reinhardt \cite{reinhardt1985absolute}. Carlson's proof, despite its significance, was limited by its dependency on a somewhat arbitrary choice of background knowledge axioms. Carlson proof, subject to but small changes, is likewise valid for various other sets of
background axioms.  The present paper examines generic
 theories for propositional modal logic. In our concluding remarks we discuss the developments of this paper in connection with work done to generalize Carlson’s consistency
result.

The paper is organized as follows.
    In Section \ref{prelimsection} we state preliminaries.
    In Section \ref{formalizationofknowerparadoxsection} we state a propositional version of the Knower Paradox: a certain theory, consisting of standard background knowledge axioms plus an axiom intended to be read as ``This sentence is known to be false,'' is inconsistent. We discuss a possible resolution to the paradox: weaken the background knowledge axioms in order to render the theory consistent.
    In Section \ref{generictheoriessection} we introduce generic and closed generic theories.
    In Section \ref{twogeneralizedconsistencystatementssection} we use generic and closed generic theories to state very generalized versions of the consistency result from Section \ref{formalizationofknowerparadoxsection}.
    In Section \ref{negativeresultssecn} we state some negative results about genericness and closed genericness. Proofs of these negative results naturally lead to the construction of exotic models which satisfy certain standard knowledge axioms while failing certain other standard knowledge axioms.
    In Section \ref{conclusionsection} we conclude the paper with a high-level discussion.
    In Appendix \ref{appendix:proofs} we give proofs of some of the claims made in the above sections.

\section{Preliminaries}
\label{prelimsection}

Throughout, we fix a nonempty set of symbols called \emph{propositional atoms}
and a symbol $\mathrm{K}$ which is not a propositional atom. The following logic is a propositional
version of Carlson's so-called \emph{base logic} \cite{carlson2000knowledge}
(cf.\ \cite{alexander2013axiomatic} and \cite{aldiniknowledge}).

\begin{definition}
\label{logicdefn} The set of \emph{formulas} is defined recursively as follows: 
\begin{itemize}
\item[(i)] Every propositional atom is a formula;
\item[(ii)] Whenever $\varphi$ and $\psi$ are formulas, so are $\neg\varphi$,
$(\varphi\wedge\psi)$, $(\varphi\vee\psi)$, and $(\varphi\rightarrow\psi)$; and
\item[(iii)] Whenever $\varphi$ is formula, so too is $\mathrm{K}(\varphi)$.
\end{itemize}
A formula is said to be \emph{basic} if it is either a propositional atom or a formula of the form $\mathrm{K}(\varphi)$ for some formula $\varphi$.  A set of formulas is called a \emph{theory}.\qeddef
        \end{definition}

We adopt  standard conventions for omitting parentheses.  Parentheses omitted from conditional formulas  are assumed to be right-nested; thus, for example, we write $\phi\rightarrow\psi\rightarrow\rho$ for 
$\phi\rightarrow (\psi\rightarrow\rho)$, and similarly for
longer chains of implications.

\begin{definition} A \emph{model} is a function mapping each basic formula to a truth value in $\{\mbox{True},\mbox{False}\}$.\qeddef
\end{definition}
Thus, in contrast with classical treatments of semantics for modalities, a model assigns truth values not only to propositional atoms but also to formulas prefixed with $\mathrm{K}$.

 We may define a binary relation $\models$ from models to basic formulas in the usual way --- that is, by stipulating that $\mathscr{M}\models\varphi$ just in case $\mathscr{M}$ assigns to $\varphi$ the value \mbox{True}. The next definition extends this relation to all formulas. We adopt the standard convention to write $\mathscr{M}\not\models\varphi$ if it is not the case that $\mathscr{M}\models\varphi$.

\begin{definition} Let $\mathscr{M}$ be a model. Define formula $\varphi$ to be \emph{true} in $\mathscr{M}$, $\mathscr{M}\models\varphi$, by recursion on $\varphi$:
\begin{itemize}
\item[(i)] If $\varphi$ is a basic formula, then $\mathscr M\models\varphi$ if and only if
$\mathscr M$ assigns to $\varphi$ the value $\mbox{True}$;
\item[(ii)]
$\mathscr M\models\neg\varphi$ if and only if $\mathscr M\not\models\varphi$;
\item[(iii)]
$\mathscr M\models\varphi\wedge\psi$ if and only if both $\mathscr M\models\varphi$ and
    $\mathscr M\models\psi$; 
\item[(iv)]
$\mathscr M\models\varphi\vee\psi$ if and only if either $\mathscr M\models\varphi$ or
    $\mathscr M\models\psi$; and
\item[(v)]
$\mathscr M\models\varphi\rightarrow\psi$ if and only if either $\mathscr M\not\models\varphi$
    or $\mathscr M\models\psi$.
\end{itemize}
Given a theory $T$, we write $\mathscr M\models T$ just in case  $\mathscr M\models\varphi$ for every $\varphi\in T$.\qeddef
\end{definition}

Entailment and validity  are given standard treatment. 

\begin{definition}
\label{validitydefn}
A theory $T$ is said to \emph{entail} a formula $\varphi$, written $T\models\varphi$, if
for all models $\mathscr M$, $\mathscr M\models T$ implies $\mathscr M\models\varphi$. A formula $\varphi$ is said to be \emph{valid}, written $\models\varphi$, if $\emptyset\models \varphi$.\qeddef
\end{definition}
Since 
modal formulas of the form $\mathrm{K}\varphi$ are treated like propositional atoms, it follows that if $p$ is a propositional atom, then $\mathrm{K}p\vee \neg \mathrm{K}p$ is valid but
$\mathrm{K}(p\vee\neg p)$ is not.
Routine argument establishes compactness. A useful result is the following corollary of compactness.

\begin{restatable}{lemma}{compactness}
\label{compactness}
Let $T$ be a theory and $\varphi$ be a formula.  Then $T\models\varphi$ if and only if
    there is a finite sequence of formulas $\varphi_1,\ldots,\varphi_n\in T$ for which
    $\models\varphi_1\rightarrow\cdots\rightarrow\varphi_n\rightarrow\varphi$.
\end{restatable}

 Lemma \ref{compactness} provides a basis for adopting the following proof-theoretic terminology in what follows.

\begin{definition} A theory $T$ is said to be \emph{consistent} if there is a model $\mathscr{M}$ for which $\mathscr{M}\models T$. \qeddef
\end{definition}

The following definition captures the familiar notion of closedness under the $\mathrm{K}$ operator.

\begin{definition}
    A theory $T$ is \emph{closed} if $\bigl\{\,\mathrm{K}\varphi\;:\;\varphi\in T\,\bigr\}\subseteq T$.   \qeddef
\end{definition}
 Thus a theory $T$  is closed just in case for every formula $\varphi$, if $\varphi\in T$, then $\mathrm{K}\varphi\in T$.
\begin{definition}
     We adopt the following conventions for naming standard schemas:
    \begin{itemize}
        \item[]
        $\mathbf{V}$ is the theory consisting of all formulas of the form $\mathrm{K}\varphi$ such that
            $\varphi$ is valid (Definition \ref{validitydefn}).
        \item[]
        $\mathbf{K}$ is the theory consisting of all formulas of the form
            $\mathrm{K}(\varphi\rightarrow\psi)\rightarrow (\mathrm{K}\varphi\rightarrow \mathrm{K}\psi)$.
        \item[]
        $\mathbf{T}$ is the theory consisting of all formulas of the form
            $\mathrm{K}\varphi\rightarrow\varphi$.
        \item[]
        $\mathbf{KK}$ (sometimes also called $\mathbf{4}$) is the theory consisting of all formulas of the form
            $\mathrm{K}\varphi\rightarrow \mathrm{K}\mathrm{K}\varphi$.\qeddef
    \end{itemize}
\end{definition}
We conclude this section with an observation about necessitation
(proved in Appendix \ref{appendix:proofs}).

\begin{restatable}{lemma}{necessitation}{\textup{(Simulated Necessitation)}}
\label{necessitation}
    Let $T$ be a closed theory. If $T$ includes both
        $\mathbf{V}$ and $\mathbf{K}$, then for every formula $\varphi:$\;
        if $T\models\varphi$, then $T\models \mathrm{K}\varphi$.
\end{restatable}

\section{A Formalization of the Knower Paradox}
\label{formalizationofknowerparadoxsection}

We will use a propositional version of the well-known Knower Paradox \cite{kaplan1960paradox} to illustrate the ideas
of this paper. The paradox is usually formalized in first-order modal logic, where appeal to G\"odel's Diagonal Lemma admits construction of the problematic sentence without having to assume it as an axiom. In our propositional version, we instead assume the problematic sentence axiomatically, allowing us to focus on the epistemological contents of the paradox without arithmetical distractions. 

\begin{theorem}[The Knower Paradox]
\label{knowerparadox}
    Let $p$ be some propositional atom.
    Let $T_{KP}$ be the smallest closed theory
    which contains$:$
    \begin{itemize}
        \item[\textup{(i)}]
        $\mathbf V$, $\mathbf K$, and $\mathbf T$
        \item[\textup{(ii)}]
        $p\leftrightarrow \mathrm{K}\neg p$\hspace{10em} ``This sentence is known to be false''
    \end{itemize}
   Then the theory $T_{KP}$ is inconsistent.\qedthm
\end{theorem}

\begin{proof}
From schema  $\mathbf{T}$ and axiom (ii), it follows that  $T_{KP}\models \neg p$ and therefore $T_{KP}\models \mathrm{K}\neg p$ by Lemma \ref{necessitation}, whence $T_{KP}\models  p$ by axiom (ii). Hence, $T_{KP}$ is inconsistent.
\end{proof}

The next theorem provides one way the theory
in Theorem \ref{knowerparadox} may be weakened in order to restore consistency
(and so constitutes a candidate for resolving the paradox, in the sense of
Haack \cite{haack1978philosophy} or Chow \cite{chow1998surprise}).

\begin{theorem}
\label{specialcase}
    Let $p$ be some propositional atom.
    Inductively, let $(T^-_{KP})_0$ be the smallest closed theory which contains$:$
 \begin{itemize}
        \item[\textup{(i)}]
        $\mathbf V$ and $\mathbf K$
        \item[\textup{(ii)}]
        $p\leftrightarrow \mathrm{K}\neg p$\hspace{10em} ``This sentence is known to be false''
    \end{itemize}
In addition, let $T^-_{KP}$ be the theory which contains$:$
    \begin{itemize}
        \item[\textup{(a)}]
        $(T^-_{KP})_0$.
        \item[\textup{(b)}] $\mathbf T$.
    \end{itemize}
  Then theory  $T^-_{KP}$ is consistent.\qedthm
\end{theorem}

Observe that the Knower Paradox (Theorem \ref{knowerparadox}), so
formalized,
rests on the assumption that the knower know 
its own truthfulness. The key difference between $T_{KP}$
and $T^-_{KP}$ is that,
while the schema $\mathrm{K}\varphi\rightarrow\varphi$ is included in both theories,
only $T_{KP}$ includes the schema $\mathrm{K}(\mathrm{K}\varphi\rightarrow\varphi)$.
Some treatments\footnote{See \cite{cross2001paradox, cross2012paradox} for an exception.} of the Knower Paradox do not explicitly
include $\mathrm{K}(\mathrm{K}\varphi\rightarrow\varphi)$ as an assumption at all, instead including
$\mathrm{K}\varphi\rightarrow\varphi$ and using a logic where the
\emph{rule of necessitation} holds---the rule permitting one to conclude $T\models\mathrm{K}\varphi$ from
$T\models\varphi$. In such logics, if $T$ contains the schema $\mathrm{K}\varphi\rightarrow\varphi$,
then trivially $T\models \mathrm{K}\varphi\rightarrow\varphi$, so by necessitation,
$T\models \mathrm{K}(\mathrm{K}\varphi\rightarrow\varphi)$. Thus, $\mathrm{K}(\mathrm{K}\varphi\rightarrow\varphi)$ sneaks in
implicitly, in such logics.

The logic (Definition \ref{logicdefn}) studied in this paper does not presume
the rule of necessitation. The rule of necessitation can be simulated in our logic
by using Lemma \ref{necessitation}, but only if the Lemma's conditions are met---which,
in the case of $T^-_{KP}$, they are not, as $T^-_{KP}$ is not closed.
Thus, it becomes possible to weaken knowledge-of-factivity without weakening factivity itself. Theorem \ref{specialcase} shows that doing so is
one possible resolution, in the sense of
Haack \cite{haack1978philosophy} or Chow \cite{chow1998surprise}, to the paradox.\footnote{The same technique has been used to resolve (in Haack's or Chow's sense) a version of the surprise exam paradox \cite{aldiniknowledge};
        to resolve a version of
        Fitch's paradox \cite{alexander2013axiomatic}; and to
        construct a machine that knows its own code \cite{alexander2014machine}.
        Aldini et al suggest
        \cite{aldiniknowledge} it might be possible to \emph{simultaneously} resolve multiple paradoxes at once by dropping
        $\mathrm{K}(\mathrm{K}\varphi\rightarrow\varphi)$, i.e., the \emph{union} of \emph{multiple} paradoxically inconsistent theories might be consistent when so weakened.} See \cite{aldiniknowledge, stjernberg2009restricting} for discussion about the weakening of knowledge-of-factivity.
Note that this requires departing from Kripke semantics, as
the rule of necessitation always holds in Kripke semantics.

Rather than prove Theorem \ref{specialcase} directly, we will
(in Section \ref{twogeneralizedconsistencystatementssection})
prove a pair of more general theorems,
and Theorem \ref{specialcase} is a special case of either one of them.
In order to state the more general theorems, we need to first introduce certain notions of
genericity.

\section{Generic and Closed Generic Theories}
\label{generictheoriessection}

The following definition is a variant of Carlson's concept of a \emph{knowing entity} 
\cite{carlson2000knowledge}.

\begin{definition}
    Let $T$ be a theory, and let $S$ be a set of propositional atoms.
    Let $\mathscr M_{T,S}$ be the model defined by stipulating:
    \begin{itemize}
        \item[\textup{(i)}]  For any propositional atom $p$:\;
       $\mathscr M_{T,S}\models p$ if and only if $p\in S$; and
        \item[\textup{(ii)}] For any formula of the form $\mathrm{K}\varphi$:\; $\mathscr M_{T,S}\models \mathrm{K}\varphi$
            if and only if $T\models\varphi$.\qeddef
    \end{itemize}
\end{definition}
The model $\mathscr M_{T,S}$ may be loosely interpreted to be that of an agent who knows exactly the consequences of theory $T$ in a world in which all propositions from $S$ are true. We will see that these models 
are useful for establishing consistency results. 

The following definition strengthens the notion of consistency.

\begin{definition}  A theory $T$ is said to be \emph{generic}
(resp.\ \emph{closed generic})
   if for each set $S$ of propositional atoms and
   each theory (resp.\ \emph{closed theory})
   $T'$:\; if $T'\supseteq T$, then
        $\mathscr M_{T',S}\models T$.\qeddef
\end{definition}

 A theory $T$ is generic
 when $T$ is known regardless of contingent facts $S$ and however theoretical knowledge might grow in conjunction with them. Generic theories are theories that
cannot be made false by the addition of more information.

We catalogue basic properties of genericity.
\begin{proposition} Genericity enjoys the following properties$:$
\label{basicfacts}
    \begin{itemize}
        \item[\textup{(1)}]
        Unions of generic theories are generic$;$
        \item[\textup{(2)}]
        Unions of closed generic theories are closed generic$;$
        \item[\textup{(3)}]
        Every generic theory is closed generic$;$
        \item[\textup{(4)}] $\mathbf V$ is generic\textup{; and}
        \item[\textup{(5)}] $\mathbf K$ is generic.\qedthm
    \end{itemize}
\end{proposition}

\begin{proof}
    Properties (1)--(3) are readily verified.
\begin{itemize}
    \item[(4)] Let $S$ be a set of propositional atoms and let $T'\supseteq \mathbf V$.
        Let $\varphi\in \mathbf V$, we must show $\mathscr M_{T',S}\models \varphi$.
        By definition of $\mathbf V$, $\varphi$ is $\mathrm{K}\psi$ for some valid $\psi$.
        Since $\psi$ is valid, $T'\models\psi$.
        Thus $\mathscr M_{T',S}\models \mathrm{K}\psi$, as desired.

   \item[(5)] Let $S$ be a set of propositional atoms and let $T'\supseteq \mathbf K$.
        Let $\varphi\in\mathbf K$, we must show $\mathscr M_{T',S}\models\varphi$.
        By definition of $\mathbf K$, $\varphi$ is
        $\mathrm{K}(\psi\rightarrow\rho)\rightarrow (\mathrm{K}\psi\rightarrow \mathrm{K}\rho)$ for some
        $\psi$ and $\rho$.
        Assume $\mathscr M_{T',S}\models \mathrm{K}(\psi\rightarrow\rho)$
        and $\mathscr M_{T',S}\models \mathrm{K}\psi$.
        This means $T'\models \psi\rightarrow\rho$ and $T'\models \psi$.
        By modus ponens, $T'\models \rho$. So $\mathscr M_{T',S}\models \mathrm{K}\rho$,
        as desired.\qedhere
        \end{itemize}
\end{proof}

\begin{lemma}
\label{kkclosedgeneric}
    The theory $\mathbf V\cup\mathbf K\cup\mathbf{KK}$ is closed generic.\qedthm
\end{lemma}

\begin{proof}
    Let $T=\mathbf V\cup\mathbf K\cup\mathbf{KK}$.
    Let $S$ be a set of propositional atoms and let $T'\supseteq T$ be closed.
    Let $\varphi\in T$, we must show $\mathscr M_{T',S}\models\varphi$. Consider two cases:
\begin{itemize}[leftmargin=5em]
    \item[Case 1] $\varphi\in \mathbf V\cup\mathbf K$. Then $\mathscr M_{T',S}\models\varphi$ because
        $\mathbf V\cup\mathbf K$ is generic by Proposition \ref{basicfacts}, parts (1), (4), and (5).

    \item[Case 2] $\varphi\in \mathbf{KK}$. Then $\varphi$ is $\mathrm{K}\psi\rightarrow \mathrm{K}\mathrm{K}\psi$ for some $\psi$.
        Assume $\mathscr M_{T',S}\models \mathrm{K}\psi$.
        This means $T'\models \psi$.
        Since $T'$ contains $\mathbf V$ and $\mathbf K$ and is closed,
        we may simulate necessitation: Lemma \ref{necessitation}
        implies $T'\models \mathrm{K}\psi$.
        Thus $\mathscr M_{T',S}\models \mathrm{K}\mathrm{K}\psi$, as desired.\qedhere
        \end{itemize}
\end{proof}

\begin{lemma}
\label{closurelemma} Let $T_{0}$ be a theory, and let $T$ be the smallest closed theory including theory $T_0$.  Suppose theory $T_0$ is (closed) generic. Then $T$ is (closed) generic. \qedthm
\end{lemma}

\begin{proof}
    Let $S$ be a set of propositional atoms and let $T'$ be a theory (resp.\ closed theory)
    such that $T'\supseteq T$.
    Let $\varphi\in T$, we must show $\mathscr M_{T',S}\models \varphi$. Consider two cases:
\begin{itemize}[leftmargin=5em]
    \item[Case 1] $\varphi\in T_0$. Then $\mathscr M_{T',S}\models\varphi$ because $T_0$ is
    generic (resp.\ closed generic).

\item[Case 2] $\varphi\not\in T_0$.
    The only other way for $\varphi$ to be in $T$ (besides being in $T_0$) is
    by way of the closure of $T$. So $\varphi$ is $K\psi$ for some $\psi\in T$.
    Since $T'\supseteq T$ and $\psi\in T$, we have $T'\models \psi$,
    which means $\mathscr M_{T',S}\models \mathrm{K}\psi$, as desired. \qedhere
    \end{itemize}
\end{proof}
\vspace*{1.7em}
\begin{lemma}
\label{normalkripkeworkerlemma}
    Suppose $T_0$ is a generic (resp.\ closed generic) theory.
    Let $T=\{\varphi\,:\,T_0\models\varphi\}$.
    Then $T$
    is generic (resp.\ closed generic).\qedthm
\end{lemma}

\begin{proof}
    Let $S$ be an arbitrary set of propositional atoms, and let $T'$ be a theory (resp.\ closed theory) such that $T'\supseteq T$. We establish that
    $\mathscr M_{T',S}\models T$.

    For each formula $\varphi$ such that $T\models\varphi$,
    let $N(\varphi)$ be the smallest positive integer $n$ for which
    there is a sequence $\varphi_1,\ldots,\varphi_n$,
    with $\varphi_n=\varphi$,
    such that for each $i=1,\ldots,n$, either $\varphi_i\in T_0$
    or there exist $j,k<i$ such that
    $\varphi_k$ is $\varphi_j\rightarrow\varphi_i$.
    Such an $N(\varphi)$ exists by the deduction theorem.

    We prove by induction on $N(\varphi)$ that
    for every $\varphi$ such that $T_0\models\varphi$,
    $\mathscr M_{T',S}\models \varphi$.
    \vspace*{.7em}
\begin{itemize}[leftmargin=7em]
    \item[Basis Step]
 $N(\varphi)=1$
    can clearly only hold if $\varphi\in T_0$.
    In that case, $\mathscr M_{T',S}\models\varphi$
    because $T_0$ is generic (resp.\ closed generic).

    \item[Inductive Step]
    $N(\varphi)>1$.
    If $\varphi\in T_0$, we are done as in the Base Case, but assume not.
    Let $\varphi_1,\ldots,\varphi_n$ be a sequence
    of length $n=N(\varphi)$ with the above
    properties.
    
    For each $i<n$,
    the subsequence $\varphi_1,\ldots,\varphi_i$
    is a shorter sequence (with the above properties) for $\varphi_i$, showing $N(\varphi_i)<N(\varphi)$. Thus by induction, ($*$) for each $i<n$,
    $\mathscr M_{T',S}\models\varphi_i$.
    Since $\varphi\not\in T_0$, there must be
    $j,k<n$ such that $\varphi_k$ is $\varphi_j\rightarrow\varphi_n$.
    By $*$, $\mathscr M_{T',S}\models \varphi_j$
    and $\mathscr M_{T',S}\models \varphi_k$.
    So $\mathscr M_{T',S}\models \varphi_j\rightarrow\varphi_n$. By modus ponens, $\mathscr M_{T',S}\models\varphi_n$, as desired.
    \end{itemize}
\end{proof}

We conclude this section with a result throwing light on the relationship between generic theories and  normal modal logics. The proof is immediate by
combining Lemmas \ref{basicfacts}, \ref{closurelemma},
and \ref{normalkripkeworkerlemma}.
\vspace*{1em}
\begin{theorem}
    Suppose $T_0$ is a (closed) generic theory.
    Let $T$ be the \emph{normal Kripke closure} of $T_0$, i.e., the smallest closed theory containing $T_0$, $\mathbf{V}$, $\mathbf{K}$, and with the property that $T$ contains $\phi$ whenever $T\models\phi$. Then $T$ is (closed) generic.\qedthm
\end{theorem}

\vspace*{1.7em}

\section{Two Generalized Consistency Statements}
\label{twogeneralizedconsistencystatementssection}

In what follows, we  state two theorems, each generalizing
Theorem \ref{specialcase}. One might be curious whether adding $\mathbf{KK}$
to the statement of Theorem \ref{specialcase} would make the paradox reappear.
Certainly the paradox as formulated in Theorem \ref{knowerparadox} does not use
$\mathbf{KK}$ in its proof. But what if there is some other form of the Knower's
Paradox that makes use of $\mathbf{KK}$, and what if in fact we only managed to
achieve consistency because we neglected to include $\mathbf{KK}$ among the
background axioms? We could state a separate version of Theorem \ref{specialcase}
which includes $\mathbf{KK}$ and then prove that separate version, with a proof
that is extremely similar to a proof of Theorem \ref{specialcase} itself, but
then maybe there's still some further background axiom that we are still neglecting,
and we would then have to state and prove yet a third version of the theorem.
This process might go on forever, we might never exhaustively think of all the
different background axioms that critics might insist upon.

\begin{restatable}{theorem}{firstmaintheorem}
\label{firstmaintheorem} Let $p$ be a propositional atom, and let $H$ be a generic theory. Let $(T_{KP})_0$ be the smallest closed theory containing:
    \begin{itemize}
        \item[\textup{(i)}] $H$
        \item[\textup{(ii)}] $p\leftrightarrow \mathrm{K}\neg p$ \hspace{10em} ``This sentence is known to be false''
    \end{itemize}
    In addition, let $T_{KP}$ be the theory containing:
        \begin{itemize}
        \item[\textup{(a)}] $(T_{KP})_0;\textup{ and}$
         \item[\textup{(b)}] $\mathbf T$.
    \end{itemize}
    For any set $S$ of propositional atoms, if $p\not\in S$ then
    $\mathscr M_{(T_{KP})_0,S}\models T_{KP}$. In particular,
    $T_{KP}$ is consistent.\qedthm
\end{restatable}
We prove Theorem \ref{firstmaintheorem} in Appendix \ref{appendix:proofs}.  Observe that since theory $\mathbf V\cup\mathbf K$ is generic by Proposition \ref{basicfacts}, Theorem \ref{specialcase} is a special case of Theorem \ref{firstmaintheorem}.

Now modify Theorem \ref{specialcase} by replacing $\mathbf V\cup\mathbf K$ with
$\mathbf V\cup\mathbf K\cup\mathbf{KK}$.
We could not do that using Theorem \ref{firstmaintheorem} unless we first established
that $\mathbf V\cup\mathbf K\cup\mathbf{KK}$ was generic (in fact, in the next section, we will show that $\mathbf V\cup\mathbf K\cup\mathbf{KK}$ is \emph{not} generic). We do know that
$\mathbf V\cup\mathbf K\cup\mathbf{KK}$ is closed generic (Lemma \ref{kkclosedgeneric}), so we would be done if we had a version of Theorem \ref{firstmaintheorem} involving closed generic theories.

\begin{theorem}
\label{secondmaintheorem}
    Same as Theorem \ref{firstmaintheorem} but with ``generic'' replaced by
    ``closed generic.''\qedthm
\end{theorem}

A proof similar to the one for Theorem \ref{firstmaintheorem} establishes Theorem \ref{secondmaintheorem}.

\section{Negative Results about Genericness}
\label{negativeresultssecn}

We have established theory $\mathbf V\cup\mathbf K\cup \mathbf{KK}$ to be closed generic.
Are these results preserved if one or more of the arguments to the union is dropped?
For example,
is theory $\mathbf V\cup\mathbf{KK}$ closed generic? Or the theory $\mathbf K\cup\mathbf{KK}$?
What about the theory $\mathbf{KK}$  alone?
Similarly, can we strengthen  closed genericity of
$\mathbf V\cup\mathbf K\cup \mathbf{KK}$ to full genericity? We show each of these questions has one and the same answer: No. 

\begin{theorem}
\label{kknegative}
   The theory  $\mathbf V\cup\mathbf K\cup\mathbf{KK}$ fails to be generic.\qedthm
\end{theorem}

\begin{proof}
    Let $T=\mathbf V\cup\mathbf K\cup\mathbf{KK}$.
    Let $p$ be some propositional atom and let $T'=T\cup\{p\}$.
    We show that $\mathscr M_{T',\emptyset}\not\models \mathrm{K}p\rightarrow \mathrm{K}\mathrm{K}p$,
    whereby $\mathscr M_{T',\emptyset}\not\models\mathbf{KK}$ and so
     $\mathscr M_{T',\emptyset}\not\models T$, showing $T$ is not generic.  Clearly $T'\models p$, so $\mathscr M_{T',\emptyset}\models \mathrm{K}p$. What remains to show is that $\mathscr M_{T',\emptyset}\not\models \mathrm{K}\mathrm{K}p$ --- that is,
  $T'\not\models \mathrm{K}p$.  

To this end, inductively define models $\mathscr N_1$ and $\mathscr N_2$ simultaneously by stipulating $\mathscr N_1\models q$ and $\mathscr N_2\not\models q$ for each propositional atom $q$ and requiring
 that $\mathscr N_1$ and $\mathscr N_2$ interpret formulas $\mathrm{K}\varphi$
 in the following way:
    \begin{itemize}
        \item[] $\mathscr N_2\models \mathrm{K}\varphi$ if and only if
            $\mathscr N_2\models\varphi$; and
            \item[] $\mathscr N_1\models \mathrm{K}\varphi$ if and only if
            $\mathscr N_2\models\varphi$.
    \end{itemize}
    Since $\mathscr N_2\not\models p$, $\mathscr N_1\not\models \mathrm{K}p$.
    Thus, to show that $T'\not\models \mathrm{K}p$, and so conclude the proof,
    it suffices to show $\mathscr N_1\models T'$.

    Let $\varphi\in T'$.  Consider four cases:
\begin{itemize}[leftmargin=5em]
    \item[Case 1]  $\varphi\in\mathbf V$. Then $\varphi$ is $\mathrm{K}\varphi_0$ for some valid $\varphi_0$.
        Since $\varphi_0$ is valid, $\mathscr N_2\models \varphi_0$,
        so $\mathscr N_1\models \mathrm{K}\varphi_0$.

 \item[Case 2] $\varphi\in\mathbf K$. Then $\varphi$ has the form
        $\mathrm{K}(\psi\rightarrow\rho)\rightarrow (\mathrm{K}\psi\rightarrow \mathrm{K}\rho)$.
        Assume $\mathscr N_1\models \mathrm{K}(\psi\rightarrow\rho)$ and
        $\mathscr N_1\models \mathrm{K}\psi$.
        Then  $\mathscr N_2\models \psi\rightarrow\rho$
        and $\mathscr N_2\models \psi$. By modus ponens,
        $\mathscr N_2\models\rho$. Thus $\mathscr N_1\models \mathrm{K}\rho$, as desired.

    \item[Case 3] $\varphi\in\mathbf{KK}$. Then $\varphi$ has the form
        $\mathrm{K}\psi\rightarrow \mathrm{K}\mathrm{K}\psi$. Assume $\mathscr N_1\models \mathrm{K}\psi$.
        Then $\mathscr N_2\models \psi$, so $\mathscr N_2\models \mathrm{K}\psi$, whence $\mathscr N_1\models \mathrm{K}\mathrm{K}\psi$,
        as desired.

    \item[Case 4] $\varphi$ is $p$. Then $\mathscr N_1\models \varphi$ by construction.\qedhere
    \end{itemize}
\end{proof}
 Proposition \ref{basicfacts} can be used to establish an immediate  corollary of Theorem \ref{kknegative}.
\begin{corollary} The theories $\mathbf V\cup\mathbf{KK}$ and
$\mathbf K\cup\mathbf{KK}$ fail to be generic.\qedthm
\label{kkcorollary}
\end{corollary}

The following corollary follows from Theorem \ref{kknegative} and Lemma \ref{kkclosedgeneric}.
\begin{corollary}
    Not every closed generic theory is generic.\qedthm
\end{corollary}

\begin{restatable}{theorem}{vkknotclosedgeneric} 
\label{vkknotclosedgeneric}
 
 If $\mathbf V\cup \mathbf{KK}$ is closed generic, then there is at most one propositional atom.

\end{restatable}
The preceding theorem, like the one stated next, is proven in Appendix \ref{appendix:proofs}. 

\begin{theorem}\label{kukknotclosedgeneric}
    The theory  $\mathbf K\cup\mathbf{KK}$ is not closed generic.\qedthm
\end{theorem}

 The following corollary follows by Proposition \ref{basicfacts}.

\begin{corollary}
   The theory $\mathbf{KK}$ is not closed generic.\qedthm
\end{corollary}

The proof of Theorem \ref{kknegative} illustrates a technique common to all proofs appearing in Appendix \ref{appendix:proofs} for the results stated in this section --- each argument proceeds by constructing pathological models. Investigating negative results about genericness and closed genericness using this technique  locates sharp edges at the boundaries of modal logic: we are led to consider models where common assumptions no longer hold, such as models where $\mathbf K$ fails or where $\mathbf V$ fails.

We have applied the theory of genericity to the Knower Paradox.
In the proofs of the following theorems, we will reverse the direction of
application, applying the Knower Paradox to the theory of genericity, rather
than vice versa.

\begin{theorem}
\label{nongericityofTthm}
    The theory $\mathbf{T}$ is not closed generic.
    In fact, no superset of $\mathbf{T}$ is closed generic.\qedthm
\end{theorem}

\begin{proof}
    Assume $\mathbf{T}^+\supseteq \mathbf{T}$ is closed generic.
    By Lemma \ref{basicfacts},
    $H=\mathbf{V}\cup \mathbf{K}\cup \mathbf{T}^+$
    is closed generic.
    Let $T_{KP}$ be as in Theorem \ref{secondmaintheorem}.
    By Theorem \ref{secondmaintheorem}, $T_{KP}$ is consistent.
    But it is easy to see that $T_{KP}$ is at least as
    strong as the theory of the same name from Theorem \ref{knowerparadox}
    (the Knower Paradox), which is \emph{inconsistent}. Absurd.
\end{proof}

In particular, $S4$ is not closed generic (and thus not generic),
and the same goes for $S5$. The following theorem implies that the same
also goes for $KD45$.

\begin{theorem}
    Let $\mathbf{5}$ be the schema consisting of all formulas of the form
    $\neg \mathrm{K}\phi\rightarrow \mathrm{K}\neg \mathrm{K}\phi$.
    No superset of $\mathbf{5}$ is closed generic.\qedthm
\end{theorem}

\begin{proof}
    Similar to Theorem \ref{nongericityofTthm} by
    reformulating the Knower's Paradox using
    $\mathbf{5}$ instead of $\mathbf{T}$.
\end{proof}

\section{Discussion}
\label{conclusionsection}

There are different forms of genericity, two of which we have examined above: generic theories and closed generic theories. These forms are particularly nice because of closure under union (Proposition \ref{basicfacts} parts 1--2) and because they are simple enough that we can prove some results about them.

In future work, we intend to use closed generic theories to generalize Carlson's
consistency result \cite{carlson2000knowledge} (this is almost already done
in \cite{alexander2020self}, but not quite, because the latter paper relies on
an axiom called \emph{assigned validity} to avoid some tricky nuances, whereas
Carlson does not).

\section*{Acknowledgements}
The authors thank three anonymous reviewers as well as Rineke Verbrugge for  valuable comments and suggestions to help improve this manuscript. The authors also extend their gratitude to Alessandro Aldini, Michael Grossberg, Ali Kahn, Rohit Parikh, and Max Stinchcombe, for their generous feedback on prior drafts of this manuscript.

\appendix

\section{Proofs}
\label{appendix:proofs}

\necessitation*

\begin{proof}[Proof of Lemma \ref{necessitation}]
    By Lemma \ref{compactness}, there are $\varphi_1,\ldots,\varphi_n\in T$
    such that $\varphi_1\rightarrow\cdots\rightarrow\varphi_n\rightarrow\varphi$ is valid.
    By $\mathbf V$,
    \[
        T\models \mathrm{K}(\varphi_1\rightarrow\cdots\rightarrow\varphi_n\rightarrow\varphi).
    \]
    By repeated applications of $\mathbf K$,
    \[
        T\models \mathrm{K}(\varphi_1\rightarrow\cdots\rightarrow\varphi_n\rightarrow\varphi)
            \rightarrow \mathrm{K}\varphi_1\rightarrow\cdots\rightarrow \mathrm{K}\varphi_n
            \rightarrow \mathrm{K}\varphi.
    \]
    Since $T$ contains each $\varphi_i$, the closure of $T$ ensures
    $T$ contains each $\mathrm{K}\varphi_i$.
    Thus $T\models \mathrm{K}\varphi$.
\end{proof}

\firstmaintheorem*

\begin{proof}[Proof of Theorem \ref{firstmaintheorem}]
    Let $\varphi\in T_{KP}$, we must show $\mathscr M_{(T_{KP})_0,S}\models \varphi$.  Consider four cases:
\begin{itemize}[leftmargin=4em]
    \item[Case 1] $\varphi\in H$. Then $\mathscr M_{(T_{KP})_0,S}\models\varphi$ because $(T_{KP})_0\supseteq H$
    and $H$ is generic.

   \item[Case 2] $\varphi$ is $p\leftrightarrow \mathrm{K}\neg p$.
        Since $p\not\in S$, $\mathscr M_{(T_{KP})_0,S}\not\models p$, thus it suffices
        to show $\mathscr M_{(T_{KP})_0,S}\not\models \mathrm{K}(\neg p)$.
        Let $S'$ be a set of propositional atoms with $p\in S'$,
        and let $T_\infty$ be the set of all formulas.

   We claim $\mathscr M_{T_\infty, S'}\models (T_{KP})_0$.  To see this,
   let $\psi\in (T_{KP})_0$, we must show
    $\mathscr M_{T_\infty, S'}\models\psi$. Three subcases are to be considered:
\begin{itemize}[leftmargin=3em]
    \item[Subcase 1] $\psi\in H$. Then $\mathscr M_{T_\infty, S'}\models\psi$
        because $H$ is generic and $T_\infty\supseteq H$.

   \item[Subcase 2] $\psi$ is $p\leftrightarrow \mathrm{K}\neg p$.
        Since $p\in S'$, $\mathscr M_{T_\infty,S'}\models p$.
        And since $T_\infty$ contains all formulas, $T_\infty\models \neg p$,
        thus $\mathscr M_{T_\infty,S'}\models \mathrm{K}\neg p$.
        So $\mathscr M_{T_\infty,S'}\models\psi$.

   \item[Subcase 3] $\psi$ is $\mathrm{K}\rho$ for some $\rho$ such that $\rho\in(T_{KP})_0$.
        Since $T_\infty$ contains all formulas, $T_\infty\models \rho$,
        so $\mathscr M_{T_\infty,S'}\models \mathrm{K}\rho$.
\end{itemize}
    This shows $\mathscr M_{T_\infty, S'}\models (T_{KP})_0$. Now since $\mathscr M_{T_\infty,S'}\models (T_{KP})_0$ and $\mathscr M_{T_\infty,S'}\models p$,
    this shows $(T_{KP})_0\not\models \neg p$.
    Thus $\mathscr M_{(T_{KP})_0,S}\not\models \mathrm{K}(\neg p)$, as desired.

    \item[Case 3] $\varphi$ is $\mathrm{K}\psi$ for some $\psi$ such that $\psi\in(T_{KP})_0$.
        Since $(T_{KP})_0\models\psi$, by definition $\mathscr M_{(T_{KP})_0,S}\models \mathrm{K}\psi$.

    \item[Case 4] $\varphi\in T_{KP}\backslash (T_{KP})_0$. Then $\varphi$ is an instance of $\mathbf T$,
    i.e., $\varphi$ is $\mathrm{K}\psi\rightarrow\psi$ for some $\psi$.
    Assume $\mathscr M_{(T_{KP})_0,S}\models \mathrm{K}\psi$.
    Then $(T_{KP})_0\models\psi$. By Cases 1--3,
    $\mathscr M_{(T_{KP})_0,S}\models (T_{KP})_0$.
    Thus $\mathscr M_{(T_{KP})_0,S}\models\psi$.
    \end{itemize}
\end{proof}

\begin{definition} Given a formula $\varphi$, define $\mathrm{K}^{n}\varphi$ by recursion on $n\in\mathbb{N}$ by  
    $\mathrm{K}^0\varphi=\varphi$ and $\mathrm{K}^{n+1}\varphi=
    \mathrm{K}\mathrm{K}^n\varphi$.\qeddef
\end{definition}

\vkknotclosedgeneric*

\begin{proof}[Proof of Theorem \ref{vkknotclosedgeneric}]
    Let $T=\mathbf V\cup\mathbf{KK}$.
    Assume there exist distinct propositional atoms
    $p$ and $q$.
    Let $T'$ be the theory which contains:
    \begin{itemize}
        \item $K^n\varphi$ for all $n\in\mathbb N$ and all $\varphi\in\mathbf V$.
        \item $K^n\varphi$ for all $n\in\mathbb N$ and all $\varphi\in\mathbf{KK}$.
        \item $K^n(p\rightarrow q)$ for all $n\in\mathbb N$.
        \item $K^n p$ for all $n\in\mathbb N$.
    \end{itemize}
    Clearly $T'$ is closed and $T'\supseteq T$.
    We will show $\mathscr M_{T',\emptyset}\not\models Kq\rightarrow KKq$,
    so $\mathscr M_{T',\emptyset}\not\models T$, so $T$ is not closed generic.
    Since $T'$ contains $p$ and $p\rightarrow q$, by modus ponens $T'\models q$,
    so $\mathscr M_{T',\emptyset}\models Kq$.
    It remains only to show $\mathscr M_{T',\emptyset}\not\models KKq$,
    i.e., that $T'\not\models Kq$.

    Define models $\mathscr N_1$ and $\mathscr N_2$ inductively so that:
    \begin{itemize}
        \item For every propositional atom $a$, $\mathscr N_1\models a$.
        \item For every propositional atom $a$, $\mathscr N_2\not\models a$.
        \item For every formula $\varphi$, $\mathscr N_2\models K\varphi$ iff
            $\mathscr N_2\models\varphi$.
        \item For every formula $\varphi$, $\mathscr N_1\models K\varphi$ iff
            $\mathscr N_2\models\varphi$ or $\varphi$ is $K^np$ for some $n\in\mathbb N$.
    \end{itemize}
    Since $q$ is distinct from $p$ and $\mathscr N_2\not\models q$,
    we have $\mathscr N_1\not\models Kq$. So to show $T'\not\models Kq$ (and thus finish
    the proof), it suffices to show $\mathscr N_1\models T'$. Let $\varphi\in T'$.

    Case 1: $\varphi$ is $K^n\psi$ for some $n\in\mathbb N$ and some $\psi\in \mathbf{V}$.
        Then $\varphi$ is $K^{n+1}\psi_0$ for some valid $\psi_0$.
        Since $\psi_0$ is valid, $\mathscr N_2\models\psi_0$, and it follows
        that $\mathscr N_1\models K^{n+1}\psi_0$.

    Case 2: $\varphi$ is $K^n\psi$ for some $n\in\mathbb N$ and some $\psi\in\mathbf{KK}$.
        Then $\varphi$ is $K^n(K\rho\rightarrow KK\rho)$ for some $\rho$.
        To show $\mathscr N_1\models\varphi$, it suffices to show
        $\mathscr N_2\models K\rho\rightarrow KK\rho$.
        Assume $\mathscr N_2\models K\rho$, then by definition $\mathscr N_2\models KK\rho$,
        as desired.

    Case 3: $\varphi$ is $K^n(p\rightarrow q)$ for some $n\in\mathbb N$.
        Since $\mathscr N_2\not\models p$, we have $\mathscr N_2\models p\rightarrow q$,
        thus $\mathscr N_1\models K^n(p\rightarrow q)$.

    Case 4: $\varphi$ is $p$. Then $\mathscr N_1\models \varphi$ by definition.

    Case 5: $\varphi$ is $K^np$ for some $n>0$.
        Then $\mathscr N_1\models \varphi$ by definition.
\end{proof}

\begin{proof}[Proof of Theorem \ref{kukknotclosedgeneric}]
    Let $T=\mathbf K\cup\mathbf{KK}$. Let $T'$ be the theory consisting of:
    \begin{itemize}
        \item
        $K^n\varphi$ for all $n\in\mathbb N$ and all $\varphi\in \mathbf K$.
        \item
        $K^n\varphi$ for all $n\in\mathbb N$ and all $\varphi\in \mathbf{KK}$.
    \end{itemize}
    Clearly $T'$ is closed and $T'\supseteq T$.
    Let $p$ be a propositional atom.
    We will show 
    $\mathscr M_{T',\emptyset}\not\models K(p\vee\neg p)\rightarrow KK(p\vee\neg p)$,
    showing $\mathscr M_{T',\emptyset}\not\models T$ and thus proving $T$ is not closed
    generic.
    Clearly $T'\models p\vee\neg p$, thus
    $\mathscr M_{T',\emptyset}\models K(p\vee\neg p)$. It remains to show
    $\mathscr M_{T',\emptyset}\not\models KK(p\vee\neg p)$, i.e., that
    $T'\not\models K(p\vee\neg p)$.

    Call a formula \emph{bad} if it is either $p\vee\neg p$ or is of the form
    $\varphi_1\rightarrow\cdots\rightarrow\varphi_n\rightarrow (p\vee\neg p)$.
    Let $\mathscr N$ be the model such that:
    \begin{itemize}
        \item For every propositional atom $a$, $\mathscr N\models a$.
        \item For every formula $\varphi$, $\mathscr N\models K\varphi$ iff $\varphi$ is not bad.
    \end{itemize}
    Since $p\vee\neg p$ is bad, we have $\mathscr N\not\models K(p\vee\neg p)$.
    Thus to show $T'\not\models K(p\vee\neg p)$ (and thus finish the proof),
    it suffices to show $\mathscr N\models T'$. Let $\varphi\in T'$.

    Case 1: $\varphi\in\mathbf K$.
        Then $\varphi$ has the form $K(\psi\rightarrow\rho)\rightarrow K\psi\rightarrow K\rho$.
        Assume $\mathscr N\models K(\psi\rightarrow\rho)$ and
        $\mathscr N\models K\psi$.
        Then $\psi\rightarrow\rho$ is not bad. This implies $\rho$ is not bad,
        thus $\mathscr N\models K\rho$, as desired.

    Case 2: $\varphi\in\mathbf{KK}$.
        Then $\varphi$ has the form $K\psi\rightarrow KK\psi$.
        Clearly $K\psi$ is not bad, thus $\mathscr N\models KK\psi$,
        thus $\mathscr N\models \varphi$.

    Case 3: $\varphi$ is of the form
        $K(K(\psi\rightarrow \rho)\rightarrow K\psi\rightarrow K\rho)$.
        Clearly $K(\psi\rightarrow \rho)\rightarrow K\psi\rightarrow K\rho$
        is not bad, so $\mathscr N\models \varphi$.

    Case 4: $\varphi$ is of the form
        $K(K\psi\rightarrow KK\psi)$.
        Clearly $K\psi\rightarrow KK\psi$ is not bad, so $\mathscr N\models \varphi$.

    Case 5: $\varphi$ is $K^n\psi$ for some $\psi\in\mathbf K\cup\mathbf{KK}$
        and some $n\geq 2$.
        Then $\varphi$ has the form $KK\rho$ for some $\rho$.
        Clearly $K\rho$ is not bad, thus $\mathscr N\models KK\rho$.
\end{proof}

\nocite{*}

\bibliographystyle{eptcs}
\bibliography{prop}

\begin{thebibliography}{10}
\providecommand{\bibitemdeclare}[2]{}
\providecommand{\surnamestart}{}
\providecommand{\surnameend}{}
\providecommand{\urlprefix}{Available at }
\providecommand{\url}[1]{\texttt{#1}}
\providecommand{\href}[2]{\texttt{#2}}
\providecommand{\urlalt}[2]{\href{#1}{#2}}
\providecommand{\doi}[1]{doi:\urlalt{https://doi.org/#1}{#1}}
\providecommand{\eprint}[1]{arXiv:\urlalt{https://arxiv.org/abs/#1}{#1}}
\providecommand{\bibinfo}[2]{#2}

\bibitemdeclare{incollection}{aldiniknowledge}
\bibitem{aldiniknowledge}
\bibinfo{author}{Alessandro \surnamestart Aldini\surnameend},
  \bibinfo{author}{Samuel~A. \surnamestart Alexander\surnameend} \&
  \bibinfo{author}{Pierluigi \surnamestart Graziani\surnameend}
  (\bibinfo{year}{2022}): \emph{\bibinfo{title}{Knowledge-of-own-Factivity, the
  Definition of Surprise, and a Solution to the Surprise Examination Paradox}}.
\newblock In: {\slshape \bibinfo{booktitle}{CIFMA}},
  \bibinfo{publisher}{Springer}, pp. \bibinfo{pages}{383--399},
  \doi{10.1007/978303126236430}.

\bibitemdeclare{article}{alexander2013axiomatic}
\bibitem{alexander2013axiomatic}
\bibinfo{author}{Samuel~A. \surnamestart Alexander\surnameend}
  (\bibinfo{year}{2013}): \emph{\bibinfo{title}{An Axiomatic Version of
  {F}itch’s Paradox}}.
\newblock {\slshape \bibinfo{journal}{Synthese}}
  \bibinfo{volume}{190}(\bibinfo{number}{12}), pp. \bibinfo{pages}{2015--2020},
  \doi{10.1007/s11229-011-9954-0}.

\bibitemdeclare{article}{alexander2014machine}
\bibitem{alexander2014machine}
\bibinfo{author}{Samuel~A \surnamestart Alexander\surnameend}
  (\bibinfo{year}{2014}): \emph{\bibinfo{title}{A Machine that Knows its own
  Code}}.
\newblock {\slshape \bibinfo{journal}{Studia Logica}}, pp.
  \bibinfo{pages}{567--576}, \doi{10.1007/s11225-013-9491-6}.

\bibitemdeclare{article}{alexander2015fast}
\bibitem{alexander2015fast}
\bibinfo{author}{Samuel~A \surnamestart Alexander\surnameend}
  (\bibinfo{year}{2015}): \emph{\bibinfo{title}{Fast-Collapsing Theories}}.
\newblock {\slshape \bibinfo{journal}{Studia Logica}}
  \bibinfo{volume}{103}(\bibinfo{number}{1}), pp. \bibinfo{pages}{53--73},
  \doi{10.1007/s11225-013-9537-9}.

\bibitemdeclare{article}{alexander2020self}
\bibitem{alexander2020self}
\bibinfo{author}{Samuel~A. \surnamestart Alexander\surnameend}
  (\bibinfo{year}{2020}): \emph{\bibinfo{title}{Self-Referential Theories}}.
\newblock {\slshape \bibinfo{journal}{The Journal of Symbolic Logic}}
  \bibinfo{volume}{85}(\bibinfo{number}{4}), pp. \bibinfo{pages}{1687--1716},
  \doi{10.1017/jsl.2020.54}.

\bibitemdeclare{article}{anderson1983}
\bibitem{anderson1983}
\bibinfo{author}{C.~Anthony \surnamestart Anderson\surnameend}
  (\bibinfo{year}{1983}): \emph{\bibinfo{title}{The Paradox of the Knower}}.
\newblock {\slshape \bibinfo{journal}{The Journal of Philosophy}}
  \bibinfo{volume}{80}(\bibinfo{number}{6}), pp. \bibinfo{pages}{338--355},
  \doi{10.2307/2026335}.

\bibitemdeclare{article}{carlson2000knowledge}
\bibitem{carlson2000knowledge}
\bibinfo{author}{Timothy~J \surnamestart Carlson\surnameend}
  (\bibinfo{year}{2000}): \emph{\bibinfo{title}{Knowledge, Machines, and the
  Consistency of {R}einhardt's Strong Mechanistic Thesis}}.
\newblock {\slshape \bibinfo{journal}{Annals of Pure and Applied Logic}}
  \bibinfo{volume}{105}(\bibinfo{number}{1-3}), pp. \bibinfo{pages}{51--82},
  \doi{10.1016/S0168-0072(99)00048-2}.

\bibitemdeclare{article}{chow1998surprise}
\bibitem{chow1998surprise}
\bibinfo{author}{Timothy~Y \surnamestart Chow\surnameend}
  (\bibinfo{year}{1998}): \emph{\bibinfo{title}{The Surprise Examination or
  Unexpected Hanging Paradox}}.
\newblock {\slshape \bibinfo{journal}{The American Mathematical Monthly}}
  \bibinfo{volume}{105}(\bibinfo{number}{1}), pp. \bibinfo{pages}{41--51},
  \doi{10.2307/2589525}.

\bibitemdeclare{article}{cross2001paradox}
\bibitem{cross2001paradox}
\bibinfo{author}{Charles~B \surnamestart Cross\surnameend}
  (\bibinfo{year}{2001}): \emph{\bibinfo{title}{The Paradox of the Knower
  without Epistemic Closure}}.
\newblock {\slshape \bibinfo{journal}{Mind}}
  \bibinfo{volume}{110}(\bibinfo{number}{438}), pp. \bibinfo{pages}{319--333},
  \doi{10.1093/mind/110.438.319}.

\bibitemdeclare{article}{cross2012paradox}
\bibitem{cross2012paradox}
\bibinfo{author}{Charles~B \surnamestart Cross\surnameend}
  (\bibinfo{year}{2012}): \emph{\bibinfo{title}{The Paradox of the Knower
  without Epistemic Closure --- Corrected}}.
\newblock {\slshape \bibinfo{journal}{Mind}}
  \bibinfo{volume}{121}(\bibinfo{number}{482}), pp. \bibinfo{pages}{457--466},
  \doi{10.1093/mind/fzs067}.

\bibitemdeclare{book}{haack1978philosophy}
\bibitem{haack1978philosophy}
\bibinfo{author}{Susan \surnamestart Haack\surnameend} (\bibinfo{year}{1978}):
  \emph{\bibinfo{title}{Philosophy of Logics}}.
\newblock \bibinfo{publisher}{Cambridge University Press},
  \doi{10.1017/CBO9780511812866}.

\bibitemdeclare{article}{kaplan1960paradox}
\bibitem{kaplan1960paradox}
\bibinfo{author}{David \surnamestart Kaplan\surnameend} \&
  \bibinfo{author}{Richard \surnamestart Montague\surnameend}
  (\bibinfo{year}{1960}): \emph{\bibinfo{title}{A Paradox Regained}}.
\newblock {\slshape \bibinfo{journal}{Notre Dame Journal of Formal Logic}}
  \bibinfo{volume}{1}(\bibinfo{number}{3}), pp. \bibinfo{pages}{79--90},
  \doi{10.1305/ndjfl/1093956549}.

\bibitemdeclare{article}{reinhardt1985absolute}
\bibitem{reinhardt1985absolute}
\bibinfo{author}{William~N \surnamestart Reinhardt\surnameend}
  (\bibinfo{year}{1985}): \emph{\bibinfo{title}{Absolute Versions of
  Incompleteness Theorems}}.
\newblock {\slshape \bibinfo{journal}{No{\^u}s}}, pp.
  \bibinfo{pages}{317--346}, \doi{10.2307/2214945}.

\bibitemdeclare{inproceedings}{romeijn2013example}
\bibitem{romeijn2013example}
\bibinfo{author}{Jan-Willem \surnamestart Romeijn\surnameend},
  \bibinfo{author}{E~\surnamestart Pacuit\surnameend} \&
  \bibinfo{author}{AP~\surnamestart Pedersen\surnameend}
  (\bibinfo{year}{2013}): \emph{\bibinfo{title}{When is an Example a
  Counterexample?}}
\newblock In: {\slshape \bibinfo{booktitle}{Proceedings of the 14th Conference
  on Theoretical Aspects of Rationality and Knowledge: TARK 2013}},
  \bibinfo{organization}{ACM Press}, pp. \bibinfo{pages}{156--165},
  \doi{10.48550/arXiv.1310.6432}.

\bibitemdeclare{article}{stjernberg2009restricting}
\bibitem{stjernberg2009restricting}
\bibinfo{author}{Fredrik \surnamestart Stjernberg\surnameend}
  (\bibinfo{year}{2009}): \emph{\bibinfo{title}{Restricting Factiveness}}.
\newblock {\slshape \bibinfo{journal}{Philosophical Studies}}
  \bibinfo{volume}{146}, pp. \bibinfo{pages}{29--48},
  \doi{10.1007/s11098-008-9243-z}.

\bibitemdeclare{article}{de2023solutions}
\bibitem{de2023solutions}
\bibinfo{author}{Mirjam \surnamestart de~Vos\surnameend},
  \bibinfo{author}{Rineke \surnamestart Verbrugge\surnameend} \&
  \bibinfo{author}{Barteld \surnamestart Kooi\surnameend}
  (\bibinfo{year}{2023}): \emph{\bibinfo{title}{Solutions to the Knower Paradox
  in the Light of {H}aack’s Criteria}}.
\newblock {\slshape \bibinfo{journal}{Journal of Philosophical Logic}}, pp.
  \bibinfo{pages}{1--32}, \doi{10.1007/s10992-023-09699-3}.

\end{thebibliography}

\end{document}